\documentclass[10pt]{amsart}
\usepackage[utf8]{inputenc}
\usepackage{amsfonts}
\usepackage{graphics}
\usepackage[all, cmtip]{xy}
\usepackage{amsthm,amsfonts,amssymb,amsmath,amsxtra}
\usepackage[all]{xy}
\usepackage{amsmath}
\usepackage{flexisym}
\usepackage{mathrsfs}
\usepackage{amsfonts}
\usepackage{amsmath}
\usepackage[colorlinks]{hyperref}
\hypersetup{
  citecolor   = blue 
 }

\usepackage{multicol}
\numberwithin{equation}{subsection}

\newtheorem{Theorem}{Theorem}[section]
\newtheorem{Remark}[Theorem]{Remark}
\newtheorem{Remarks}[Theorem]{Remarks}

\newtheorem{Proposition}[Theorem]{Proposition}
\newtheorem{Corollary}[Theorem]{Corollary}
\newtheorem{Conjecture}[Theorem]{Conjecture}
\newtheorem{Main Theorem}[Theorem]{Main Theorem}

\usepackage[colorlinks]{hyperref}
\makeatletter
\renewcommand*\env@matrix[1][*\c@MaxMatrixCols c]{%
  \hskip -\arraycolsep
  \let\@ifnextchar\new@ifnextchar
  \array{#1}}
\makeatother

\newif\ifgrading

\gradingfalse

\newcommand{\RC}[1]{\ifgrading {{\color{black} #1 }} \fi}

\usepackage[colorlinks]{hyperref}
\hypersetup{linkcolor=black}
\usepackage{xcolor}
\usepackage{ wasysym }

\usepackage{tikz-cd}
\newcommand{\U}{{\mathcal U}}
\newcommand{\calG}{{\mathcal G}}
\newcommand{\calO}{{\mathcal O}}
\newcommand{\Mloc}{{\rm M}^{\rm loc}}
\newcommand{\Spec}{{\rm Spec \, } }
\newcommand{\wti}{\widetilde}
\newcommand{\und}{\underline}
\usepackage{xcolor}

\begin{document}

\title[integral models for unitary Shimura varieties]{Semi-stable models for some unitary Shimura varieties over ramified primes}

\author[I. Zachos]{I. Zachos}
\address{
Dept. of
Mathematics\\
Boston College\\
Chestnut Hill\\
MA 02467\\
USA}
\email{zachosi@bc.edu}

\begin{abstract}
We consider Shimura varieties associated to a unitary group of signature $(n-2,2)$. We give regular $p$-adic integral models for these varieties over odd primes $p$ which ramify in the imaginary quadratic field with level subgroup at $p$ given by the stabilizer of a selfdual lattice in the hermitian space. Our construction is given by an explicit resolution of a corresponding local model.
\end{abstract}

\maketitle

\tableofcontents

$\quad$\\

\section{Introduction}
\subsection{}
This paper is a contribution to the problem of constructing regular integral models for Shimura varieties over places of bad reduction. There are several implicit examples of constructions of such regular integral models in special cases; see, for example, work of de Jong \cite{deJ}, Genestier \cite{Ge}, Pappas \cite{P}, Faltings \cite{Fa} and the very recent work of Pappas with the author \cite{PaZa}. Here, we consider Shimura varieties associated to unitary groups of signature $(r,s)$ over an imaginary quadratic field $F_0$. These Shimura varieties are of PEL type, so they can be written as a moduli space of abelian varieties with polarization, endomorphisms and level structure. Shimura varieties have canonical models over the “reflex” number field $E$. In the cases we consider here the reflex field is the field of rational numbers $\mathbb{Q}$ if $r = s$ and $E = F_0$ otherwise.

Constructing such well-behaved integral models is an interesting and hard problem whose solution has many applications to number theory. The behavior of these depends very much on the “level subgroup”. Here, the level subgroup is the stabilizer of a selfdual lattice in the hermitian space. This stabilizer, by what follows below, is not connected when $n$ is even, so not parahoric. However, by using work of Rapoport-Zink \cite{RZbook} and Pappas \cite{P} we construct $p$-adic integral models, which have simple and explicit moduli descriptions, and are \'etale locally around each point isomorphic to certain simpler schemes the \textit{naive local models}. Inspired by the work of Pappas-Rapoport \cite{PR2} and Krämer \cite{Kr}, we consider a variation of the above moduli problem where we add in the moduli problem an additional subspace in the deRham filtration $ {\rm Fil}^0 (A) \subset H_{dR}^1(A)$ of the universal abelian variety $A$, which satisfies certain conditions. This is essentially an instance of the notion of a ``linear modification" introduced in \cite{P}. We then show that the blow-up of this model along a smooth (non Cartier) divisor produces a semistable integral model of the corresponding Shimura variety, i.e. it is regular and the irreducible components of the special fiber are smooth divisors crossing normally. We expect that our construction will find applications to the study of arithmetic intersections of special cycles and Kudla’s program. (See \cite{Zhang}, \cite{BHKR} and \cite{HLSY} for important applications of integral models of unitary Shimura varieties to number theory.)

\subsection{} To explain our results, we need to introduce some notation. We consider the group $G$ of unitary similitudes for a hermitian vector space $(W, \phi)$ of dimension $n > 3$ over an imaginary quadratic field $F_0 \subset \mathbb{C}$, and fix a conjugacy class of homomorphisms $h : \text{Res}_{\mathbb{C}/\mathbb{R}}\mathbb{G}_{m} \rightarrow G_{\mathbb{R}}$ corresponding to a Shimura datum $(G,X_h)$ of signature $(r,s)=(n-2, 2)$ (see \S \ref{Shimura}). Let us mention here that the case $(r,s)=(1,2)$, when $n=3$, was studied in \cite[4.5, 4.15]{P} (see also \cite[\S 6]{PR}).

We assume that $F_0/ \mathbb{Q}$ is ramified over $p$, where $p$ is an odd prime number. Let $F_1 = F_0 \otimes \mathbb{Q}_p$ and $V = W \otimes_{\mathbb{Q}} \mathbb{Q}_p$. We fix a square root $\pi$ of $p$ and we set $k = \overline{\mathbb{F}_p}$. We assume
that the hermitian form $\phi$ on $V$ is split, i.e that there is a basis $e_1,\dots , e_n$ such that $\phi(e_i, e_{n+1-j}) = \delta_{ij}$ for $ i, j \in \{1,\dots n\}.$

In addition, we denote by $\Lambda$ the standard lattice $O^n_{F_1}$ in $V$ and we let $\mathcal{L}$ be the
self-dual multichain consisting of $\{\pi^k \Lambda\}_{k \in \mathbb{Z}}$. Denote by $K$ the stabilizer of $\Lambda$ in $G(\mathbb{Q}_p)$ and let $\calG$ be the (smooth) group scheme of automorphisms of the polarized chain $\mathcal{L}$ over $\mathbb{Z}_p$ (see \cite[\S 1.5]{PR}). Then $\calG(\mathbb{Z}_p) =K$ and the group scheme $\calG$ has $G\otimes_{\mathbb{Z}_p}\mathbb{Q}_p $ as its generic fiber. It turns out that when $n$ is odd the stabilizer $K$ is a parahoric subgroup. When $n$ is even, $K$ is not a parahoric subgroup since it contains a parahoric subgroup with index 2 and the corresponding parahoric group scheme is its connected component $ K^\circ$; see \cite[\S 1.2]{PR} for more details.  

Choose also a sufficiently small compact open subgroup $K^p$ of the prime-to-$p$ finite adelic points $G({\mathbb A}_{f}^p)$ of $G$ and set $\mathbf{K}=K^pK$. The Shimura variety  ${\rm Sh}_{\mathbf{K}}(G, X)$ with complex points
 \[
 {\rm Sh}_{\mathbf{K}}(G, X)(\mathbb{C})=G(\mathbb{Q})\backslash X\times G({\mathbb A}_{f})/\mathbf{K}
 \]
is of PEL type. We set $\mathcal{O} = O_{E_v}$ where $v$ the unique prime ideal of $E$ above $(p)$. 

Next, we follow \cite[Definition 6.9]{RZbook} to define the moduli scheme $\mathcal{A}^{\rm naive}_{\mathbf{K}}$ over $\mathcal{O}$ whose generic fiber agrees with ${\rm Sh}_{\mathbf{K}}(G, X)$ (see also \S \ref{Shimura}). A point of $\mathcal{A}^{\rm naive}_{\mathbf{K}}$ with values in the $\mathcal{O} $-scheme $S$ is the isomorphism class of the following set of data $(A,\bar{\lambda}, \bar{\eta})$:
\begin{enumerate}
    \item An $\mathcal{L}$-set of abelian varieties $A = \{A_{\Lambda}\}$.
    \item A $\mathbb{Q}$-homogeneous principal polarization $\bar{\lambda}$ of the $\mathcal{L}$-set $A$.
    \item A $K^p$-level structure
    \[
\bar{\eta} : H_1 (A, {\mathbb A}_{f}^p) \simeq W \otimes  {\mathbb A}_{f}^p \, \text{ mod} \, K^p
    \]
which respects the bilinear forms on both sides up to a constant in $({\mathbb A}_{f}^p)^{\times}$ (see loc. cit. for
details).

The set $A$ should satisfy the determinant condition (i) of loc. cit.
\end{enumerate}

For the definitions of the terms employed here we refer to loc.cit., 6.3–6.8 and \cite[\S 3]{P}. The functor $\mathcal{A}^{\rm naive}_{\mathbf{K}}$ is representable by a quasi-projective scheme over $\mathcal{O}$. The moduli scheme $\mathcal{A}^{\rm naive}_{\mathbf{K}}$ is connected to the \textit{naive local model} ${\rm M}^{\rm naive}$, see \S \ref{Prelim} for the explicit definition of ${\rm M}^{\rm naive}$, via the local model diagram  
 \begin{equation}\label{LMdiagram}
\begin{tikzcd}
&\wti{\mathcal{A}}^{\rm naive}_{\mathbf{K}}\arrow[dl, "\pi^{\rm }_{\mathbf{K}}"']\arrow[dr, "q^{\rm }_{\mathbf{K}}"]  & \\
\mathcal{A}^{\rm naive}_{\mathbf{K}} &&  {\rm M}^{\rm naive}
\end{tikzcd}
\end{equation}
where the morphism $\pi^{\rm }_{\mathbf{K}}$ is a $\mathcal{G}$-torsor and $q^{\rm }_{\mathbf{K}}$ is a smooth and $\mathcal{G}$-equivariant morphism (see \S \ref{Shimura}). Equivalently,
using the language of algebraic stacks, there is a relatively representable smooth morphism
 \[ \mathcal{A}^{\rm naive}_{\mathbf{K}} \to [\mathcal{G} \backslash  {\rm M}^{\rm naive}] \]
where the target is the quotient algebraic stack. In particular, since $\calG$ is smooth, the above imply that $\mathcal{A}^{\rm naive}_{\mathbf{K}} $ is \'etale locally isomorphic to ${\rm M}^{\rm naive}$. 

One can now consider a variation of the moduli of abelian schemes $\mathcal{A}^{\rm spl}_{\mathbf{K}} $ over $\Spec O_{F_1}$ where we add in the moduli problem an additional subspace in the Hodge filtration $ {\rm Fil}^0 (A) \subset H_{dR}^1(A)$ of the universal abelian variety $A$ with certain conditions to imitate the definition of the splitting local model $\mathcal{M}$; see \S \ref{un.int.} for the explicit definition of $\mathcal{A}^{\rm spl}_{\mathbf{K}} $ and \S \ref{Prelim} where we define $\mathcal{M}$ for general signature $(r,s)$. (Actually, $\mathcal{M}$ is a generalization of Krämer's local models \cite[Definition 4.1]{Kr}). There is a forgetful morphism
\[
\tau :   \mathcal{A}^{\rm spl}_{\mathbf{K}} \longrightarrow \mathcal{A}^{\rm naive}_{\mathbf{K}} \otimes_{\mathcal{O}} O_{F_1}
\]
defined by forgetting the extra subspace. Moreover,  $\mathcal{A}^{\rm spl}_{\mathbf{K}}$ has the same \'etale local structure as $\mathcal{M}$ and is a linear modification of $\mathcal{A}^{\rm naive}_{\mathbf{K}} \otimes_{\mathcal{O}} O_{F_1}$ in the sense of \cite[\S 2]{P} (see also \cite[\S 15]{PR2}). Therefore, there is a local model diagram for $\mathcal{A}^{\rm spl}_{\mathbf{K}}$ similar to (\ref{LMdiagram}) but with $ {\rm M}^{\rm naive} $ replaced by $ \mathcal{M}$. Note, that there is also a corresponding forgetful morphism
\[
\tau_1 :   \mathcal{M} \longrightarrow  {\rm M}^{\rm naive} \otimes_{\mathcal{O}} O_{F_1}.
\]

In \S \ref{Prelim}, we show that $\tau_1^{-1}(*)$ is isomorphic to the Grassmannian $Gr(2,n)_k$. Here, $*$ is the ``worst point" of ${\rm M}^{\rm naive}$, i.e.  the  unique closed $\mathcal{G}$-orbit supported in the special fiber; see \cite[\S 4]{P} for more details. Under the local model diagram, (see \S \ref{Shimura}), $\tau_1^{-1}(*)$ corresponds to the locus where 
the Hodge filtration ${\rm Fil}^0 (A) $ of the universal abelian scheme $A$ is annihilated by the action of the uniformizer $\pi$. Consider the blow-up $\mathcal{A}^{\rm 
bl}_{\mathbf{K}}$ of $\mathcal{A}^{\rm spl}_{\mathbf{K}}$ along this locus.

\subsection{} The main result of the paper is the following theorem.
\begin{Theorem}
$\mathcal{A}^{\rm bl}_{\mathbf{K}}$ is a semi-stable integral model for the Shimura variety ${\rm Sh}_{\mathbf{K}}(G, X)$.
\end{Theorem}
Since  blowing-up commutes with \'etale localization and the \'etale local structure of the moduli scheme $\mathcal{A}^{\rm spl}_{\mathbf{K}}$ is controlled by the local structure of the local model $\mathcal{M}$, it is enough to show the above statement for the corresponding local models. In particular, it suffices to prove:
\begin{Theorem}
The blow-up ${\rm M}^{\rm bl}$ of $\mathcal{M} $ along the smooth irreducible component $\tau_1^{-1}(*) $ of its special fiber is regular and has special fiber a divisor with normal crossings.
\end{Theorem}

To show the above theorem, we explicitly calculate an affine chart $\U$ of $\mathcal{M}$ in a neighbourhood of $  \tau_1^{-1}(*)$. In fact, we consider a more general situation where we calculate $\U$ for a general signature $(r,s)$ and we show that $\mathcal{G}$-translates of $\U$ cover $\mathcal{M}$.

\begin{Proposition}\label{redeqIntro}
An affine chart $\U \subset \mathcal{M}$ containing a preimage of the worst point is isomorphic to 
\[
\Spec  O_{F_1}[X , Y ]/(X-X^t,\, X\cdot(I_s+Y^t \cdot Y)-2\pi I_s )
\]
where $X,Y$ are of sizes $s\times s$ and $(n-s) \times s $ respectively. 
\end{Proposition}
When $(r,s) = (n-1,1)$, Krämer \cite{Kr} shows that $\U$, and so $\mathcal{M}$, has semi-stable reduction. Therefore, she obtains a semistable integral model for the corresponding Shimura variety. 

When $(r,s) = (n-2,2)$, $\U $ does not have semi-stable reduction anymore and so $\mathcal{M}$ does not give us a resolution. However, we use the explicit description of $\U$ above to calculate the blow-up of $\mathcal{M}$ along the $\mathcal{G}$-invariant smooth subscheme $\tau_1^{-1}(*)$. The blow-up gives a $\calG$-birational projective morphism 
\[\
r^{\rm bl} : {\rm M}^{\rm bl} \rightarrow \mathcal{M}
\]
such that ${\rm M}^{\rm bl}$ is regular and has special fiber a reduced divisor with normal crossings. We quickly see that the corresponding blow-up $\mathcal{A}^{\rm bl}_{\mathbf{K}}$ of the integral model $\mathcal{A}^{\rm spl}_{\mathbf{K}}$ inherits the same nice properties as $ {\rm M}^{\rm bl} $. In fact, there is a local model diagram for $\mathcal{A}^{\rm bl}_{\mathbf{K}}$ similar to (\ref{LMdiagram}) but with $ {\rm M}^{\rm naive} $ replaced by $ {\rm M}^{\rm bl} $. See Theorem \ref{RegLM} for the precise statement about the model $\mathcal{A}^{\rm bl}_{\mathbf{K}}$.

Let us mention here that we can obtain similar results for the Shimura varieties ${\rm Sh}_{\mathbf{K}'}(G, X)$ where $\mathbf{K}' = K^p K^\circ$ (see \S \ref{Shimura}). (Recall that $ K^\circ$ is the parahoric connected component of the stabilizer $K$.) Also, we can apply these results to obtain regular (formal) models of the corresponding Rapoport-Zink spaces.

Let us now explain the lay-out of the paper. In \S \ref{Prelim}, we recall the definitions of certain variants of local models for ramified unitary
groups. In \S \ref{AffineChart}, we give explicit equations that describe the affine chart $\U$ of the \textit{splitting model} $ \mathcal{M}$ for a general signature $(r,s)$ and we also show that $\mathcal{G}$-translates of $\U$ cover $\mathcal{M}$. For the rest of the paper we assume $(r,s)= (n-2,2)$. In \S \ref{BlowUp}, we  construct the semi-stable resolution $\rho: \U^{\rm bl} \rightarrow \U$ of the affine chart $\U$. In \S \ref{Resol}, we show that ${\rm M}^{\rm bl}$ has semi-stable reduction by using the results of \S \ref{BlowUp} and the structure of local models. In \S \ref{Shimura}, we apply the above results to construct regular integral models for the corresponding Shimura varieties. 
\smallskip

{\bf Acknowledgements:} I thank G. Pappas for his valuable comments, insights and corrections on a preliminary version of this article. I also thank B. Howard and K. Madapusi Pera for useful suggestions and the referee for their careful reading that lead to several corrections.

\section{Preliminaries: local models and variants}\label{Prelim}

We use the notation of \cite{P}. We take $F =\mathbb{Q}_p[t]/(t^2 - pu) $ and $O_F = \mathbb{Z}_p[t]/(t^2 - pu)$, where $p$ is an odd prime and $u$ is a unit in $\mathbb{Z}_p$. For $n > 3$, we set $V = F^n$  and denote by $e_i$, $1 \leq  i \leq n$, the
standard $O_F$-generators of the standard lattice $\Lambda = O_F^{n} \subset V.$ Fix a uniformizer $\pi$ of $O_F$ with $\pi^2= p\delta$. Also, since $p \neq 2$, $ \delta = \pi^2/p$ has a square root in
a finite \'etale extension of $\Spec (\mathbb{Z}_p)$. After such a base extension there is
a uniformizer $\pi$ such that $\pi^2= p$. We will assume that we have such a
uniformizer and suppress the notation of the \'etale base extension.

Set $k = \overline{\mathbb{F}_p}$. The uniformizing element $\pi$ induces a $\mathbb{Z}_p$- linear mapping on $\Lambda$ which we denote by $t$. We define a non-degenerate alternating $\mathbb{Q}_p$-bilinear form $\langle\, , \, \rangle : V \times V \rightarrow \mathbb{Q}_p$ given by
$$
\langle e_i
, t e_j \rangle = \delta_{i,j} , \quad \langle e_i
, e_j \rangle = 0, \quad \langle te_i
, te_j \rangle = 0.
$$
The restriction $\langle\, , \, \rangle : O_F^{n} \times O_F^{n} \rightarrow \mathbb{Z}_p$ is a perfect $\mathbb{Z}_p$-bilinear form.  Using the duality isomorphism $\text{Hom}_F (V, F) \cong \text{Hom}_{\mathbb{Q}_p}(V,\mathbb{Q}_p)$ given by
composing with the trace $Tr_{F/\mathbb{Q}_p} :  F \rightarrow \mathbb{Q}_p $ we see, as in \cite[\S 3]{P}, that there exists a unique nondegenerate
hermitian form $\phi : V \times V \rightarrow F$ such that
\[
\langle x, y \rangle = Tr_{F /\mathbb{Q}_p} (\pi^{-1}\phi(x, y)), \quad x, \, y \in V.
\]
We take $G :=GU_n:= GU(\phi)$ and we choose a partition $n = r + s$; we refer to the pair $(r, s)$ as the signature. By replacing $\phi$ by $-\phi$ if needed, we can make sure that $s\leq r$ and so we assume that $s \leq r$ (see \cite[\S 1.1]{PR} for more details). Identifying $G \otimes F \simeq GL_{n,F} \times \mathbb{G}_{m,F}$, we define the cocharacter $ \mu_{r,s}$ as $ (1^{(s)}, 0^{(r)}, 1)$ of $D \times \mathbb{G}_m$, where $D$ is the standard maximal torus of diagonal matrices in $GL_n$; for more details we refer the reader to \cite[\S 3.2]{Sm1}. We denote by $E$ the reflex field of $\{ \mu_{r,s}\}$; then $E = \mathbb{Q}_p$ if $r = s$ and $E = F$ otherwise (see \cite[\S 1.1]{PR}). We set $O := O_E$.

Next, we denote by $K$ the stabilizer of $\Lambda$ in $G(\mathbb{Q}_p)$. We also let $\mathcal{L}$ be the
self-dual multichain consisting of $\{\pi^k \Lambda\}_{k \in \mathbb{Z}}$. Here $\mathcal{G} = \underline{{\rm Aut}}(\mathcal{L})$ is the group scheme over $\mathbb{Z}_p$ with $K = \mathcal{G}(\mathbb{Z}_p)$ the subgroup of $G(\mathbb{Q}_p)$ fixing the lattice chain $\mathcal{L}$. When $n$ is odd, the stabilizer $K$ is a parahoric subgroup but when $n$ is even, $K$ is not a parahoric subgroup since 
it contains a parahoric subgroup with index 2. The corresponding parahoric group scheme is its connected component $ K^\circ$. (See \cite[\S 1.2]{PR} for more details.) 
\smallskip

We briefly recall the definition of certain variants of local models for ramified unitary groups. 

\subsection{Rapoport-Zink local models and variants:} Let ${\rm M}^{\rm naive}$ be the functor which associates to each scheme $S$
over $\Spec O$ the set of subsheaves $\mathcal{F}$ of $O \otimes \mathcal{O}_S $-modules of $ \Lambda \otimes \mathcal{O}_S  $
such that
\begin{enumerate}
    \item $\mathcal{F}$ as an $\mathcal{O}_S $-module is Zariski locally on $S$ a direct summand of rank $n$;
    \item $\mathcal{F}$ is totally isotropic for $\langle \,
, \, \rangle \otimes \mathcal{O}_S$;
    \item (Kottwitz condition) $\text{char}_{t |  \mathcal{F} } (X)= (X + \pi)^r(X - \pi)^s $.
\end{enumerate}
The functor ${\rm M}^{\rm naive}$ is represented by a closed subscheme, which we again
denote ${\rm M}^{\rm naive}$, of $Gr(n, 2n) \otimes \Spec O$; hence ${\rm M}^{\rm naive}$ is a projective $ O$-scheme. (Here we denote by $Gr(n, d)$ the Grassmannian scheme parameterizing locally direct summands of rank $n$ of a free module of rank $d$.) ${\rm M}^{\rm naive}$ is the \textit{naive local model} of  Rapoport-Zink \cite{RZbook}. Also, ${\rm M}^{\rm naive}$ supports an action of $\mathcal{G}$. 
\begin{Proposition}\label{notflat}
a) We have 
\[
{\rm M}^{\rm naive} \otimes_{ O} E \cong Gr(s,n)\otimes E.
\]
In particular, the generic fiber of ${\rm M}^{\rm naive}$ is smooth and geometrically irreducible of dimension $rs$.

b) We have 
\[
\text{dim} \,  {\rm M}^{\rm naive} \otimes_{O} k  \geq \left\{
	\begin{array}{ll}
		n^2/4  & \mbox{if } \, n  \, \mbox{is even} \\
		(n^2-1)/4 &  \mbox{if } \, n  \, \mbox{is odd}.
	\end{array}
\right.
\]
In particular, ${\rm M}^{\rm naive}$ is not flat if $|r -s| > 1$.
\end{Proposition}
\begin{proof}
See \cite[Proposition 3.8]{P}, \cite[Proposition 2.2]{Kr} and \cite[Corollary 2.3]{Kr}. 
\end{proof}
The flat closure of $ {\rm M}^{\rm naive} \otimes_{O} E$ in ${\rm M}^{\rm naive}$ is by definition the \textit{local model} $\Mloc$.

In \cite[\S 4]{P}, Pappas introduces  the \textit{wedge local model} ${\rm M}^{\wedge}$, in order to correct the non-flatness problem, by imposing the following additional condition: 
\[
 \wedge^{r+1} (t-\pi | \mathcal{F}) = (0) \quad \text{and} \quad  \wedge^{s+1} (t+\pi | \mathcal{F}) = (0). 
\]
More precisely, ${\rm M}^{\wedge}$ is the closed subscheme of ${\rm M}^{\rm naive}$ that classifies points given by $\mathcal{F}$ which satisfy the wedge condition. The scheme ${\rm M}^{\wedge}$ supports an action of $\mathcal{G}$ and the immersion $i : {\rm M}^{\wedge} \rightarrow
{\rm M}^{\rm naive}$ is $\mathcal{G}$-equivariant. It is easy to see that
\begin{Proposition}
The generic fibers of ${\rm M}^{\rm naive} $ and ${\rm M}^{\wedge} $ coincide, in particular
the generic fiber of ${\rm M}^{\wedge}$ is a smooth, geometrically irreducible variety of
dimension rs.
\end{Proposition}
\begin{proof}
See  \cite[Proposition 3.4]{Kr} and \cite[Lemma 1.1]{A}. 
\end{proof}
Also, $\Mloc \subset {\rm M}^{\wedge}$ and $\Mloc \otimes E = {\rm M}^{\wedge}\otimes E $. As in \cite[\S 4]{P} and \cite[\S 2.4.2, \S 5.5]{PR}, the worst point of ${\rm M}^{\wedge} $, i.e. the  unique closed $\calG$-orbit which lies in the closure of any other orbit, is given by the $ k$-valued point $\mathcal{F} = t\Lambda \subset \Lambda\otimes k  \cong (k[t]/(t^2))^n$.

It is conjectured in \cite{P} that ${\rm M}^{\wedge}$ is flat for $n \geq 2$ and any signature and so $\Mloc = {\rm M}^{\wedge}$. It has been shown, see \cite[Theorem 4.5]{P}, that this is true for the signature $(n-1,1).$ For more general lattice chains, the wedge condition turns out to be insufficient, see \cite[Remarks. 5.3, 7.4]{PR}. In \cite{PR}, the authors introduced a further refinement of the moduli problem by also adding the so-called ``spin condition" (for more information we refer the reader to \cite{PR}); this will play no role in this paper.

Next, we consider the moduli scheme $\mathcal{M}$ over $O_F$, the \textit{splitting (or Krämer) local model} as in \cite[Remark 14.2]{PR2} and \cite[Definition 4.1]{Kr}, whose points for an $O_F $-scheme $S$ are Zariski locally $\mathcal{O}_S$-direct summands $ \mathcal{F}_0 , \mathcal{F}_1 $ of ranks $s$, $n$ respectively, such that

\begin{enumerate}
    \item $ (0) \subset \mathcal{F}_0 \subset \mathcal{F}_1 \subset \Lambda  \otimes \mathcal{O}_S$;
    \item $\mathcal{F}_1 = \mathcal{F}^{\bot}_1$, i.e. $\mathcal{F}_1$ is totally isotropic for $\langle \,
, \, \rangle \otimes \mathcal{O}_S$;
    \item $(t + \pi) \mathcal{F}_1 \subset \mathcal{F}_0$;
    \item $(t - \pi)\mathcal{F}_0 = (0)$.
\end{enumerate}
The functor is represented by a projective $O_F $-scheme $ \mathcal{M}$. The scheme $ \mathcal{M}$ supports an action of $\calG$ and there is a $\calG$-equivariant morphism 
\[
\tau : \mathcal{M} \rightarrow {\rm M}^{\wedge}\otimes_{O} O_F 
\]
which is given by $(\mathcal{F}_0,\mathcal{F}_1) \mapsto \mathcal{F}_1$ on $S$-valued points. (Indeed, we can easily see, as in \cite[Definition 4.1]{Kr}, that $\tau$ is well defined.)  
\begin{Proposition}\label{SmoothGen.M}
The morphism $\tau  : \mathcal{M} \rightarrow {\rm M}^{\wedge}\otimes_{O} O_F $ induces an isomorphism on the generic fibers.
\end{Proposition}
\begin{proof}
It follows by \cite[Remark 4.2]{Kr} and the proof of  \cite[Proposition 3.8]{P}. 
\end{proof}
The following discussion appears in \cite{PappasLetter}. Over the special fiber, the condition (4) amounts to $t\mathcal{F}_0 = (0)$. Thus, we have 
\[
(0) \subset \mathcal{F}_0 \subset t \Lambda \otimes k \subset \mathcal{F}^{\bot}_0 \subset \Lambda\otimes k .
\]
Also, we have 
\[
(0) \subset (t^{-1}(\mathcal{F}_0))^{\bot} \subset t\Lambda \otimes k \subset t^{-1}(\mathcal{F}_0) \subset \Lambda \otimes k.
\]
The spaces $t^{-1}(\mathcal{F}_0) , \,\mathcal{F}^{\bot}_0$ have rank $n + s$, $2n - s = n + r$ respectively. Fixing $\mathcal{F}_0 $, the rank of $t^{-1}(\mathcal{F}_0) \cap \mathcal{F}^{\bot}_0 $ influences the dimension of the space of allowable $\mathcal{F}_1$ since
\[
\mathcal{F}_0 + (t^{-1}(\mathcal{F}_0))^{\bot}  \subset \mathcal{F}_1 \subset t^{-1}(\mathcal{F}_0) \cap \mathcal{F}^{\bot}_0.
\]
Note that $\mathcal{F}_0 \subset t \Lambda\otimes k \simeq k^n \otimes \mathcal{O}_S $. Hence, we consider the morphism 
\[
\pi : \mathcal{M}\otimes k \rightarrow Gr(s,n)\otimes k 
\]
given by $ (\mathcal{F}_0,\mathcal{F}_1) \mapsto \mathcal{F}_0.$ This has a section 
\[
\phi : Gr(s,n)\otimes k \rightarrow \mathcal{M}\otimes k
\]
given by $ \mathcal{F}_0 \mapsto (\mathcal{F}_0,\mathcal{F}_1)$ with $\mathcal{F}_1 = t\Lambda \otimes k.$ The image of the section $\phi$ is an irreducible component of $ \mathcal{M}\otimes_O k$ which is the fiber $\tau^{-1} (t\Lambda)$ over the worst point. Hence, $\tau^{-1}(t\Lambda)$ is isomorphic to the Grassmannian $ Gr(s,n)\otimes k  $ of dimension $rs$. Also, observe that $\{ (\mathcal{F}_0, \mathcal{F}_1) \, | \,\text{rank} \, (t^{-1}(\mathcal{F}_0) \cap \mathcal{F}^{\bot}_0) = n \} \subset \tau^{-1}(t\Lambda) $ since when $t^{-1}(\mathcal{F}_0) \cap \mathcal{F}^{\bot}_0$ has rank $n$ then $ t^{-1}(\mathcal{F}_0) \cap \mathcal{F}^{\bot}_0 = t\Lambda$ which gives $\mathcal{F}_1 = t\Lambda.$

However, the morphism $\pi$ has fibers of positive dimension over points of $Gr(s,n)\otimes k $ which correspond to subspaces of $Gr(s,n)\otimes k$ on which the dimension $t^{-1}(\mathcal{F}_0) \cap \mathcal{F}^{\bot}_0$ is more than $n$. Actually, $t^{-1}(\mathcal{F}_0) \cap \mathcal{F}^{\bot}_0$ has maximal dimension, i.e. $ t^{-1}(\mathcal{F}_0) \subset \mathcal{F}_0^{\bot}$, if and only if $\mathcal{F}_0 \subset t \Lambda\otimes k \simeq k^n \otimes \mathcal{O}_S $ is totally isotropic for the (non-degenerate)
symmetric form on $ t \Lambda\otimes k $ which is defined as $\{tv,tw \} := \langle tv,w \rangle$; see the proof of \cite[Theorem 4.5]{Kr} for more details. Denote by $\mathcal{Q}(s,n)$ the smooth closed subscheme of $Gr(s,n)\otimes k $ of dimension $ s(2n-3s-1)/2$ which parametrizes all these isotropic $s$-subspaces $\mathcal{F}_0$ in the $n$-space $k^n \otimes \mathcal{O}_S$. For such $ \mathcal{F}_0 \in \mathcal{Q}(s,n)$ we have that $ t^{-1}(\mathcal{F}_0) \subset \mathcal{F}_0^{\bot}$ and thus the fiber $\pi ^{-1}(\mathcal{F}_0)$ is given by $\mathcal{F}_1$ with $\mathcal{F}_1 = \mathcal{F}^{\bot}_1$ such that
\[
\mathcal{F}_0 \subset  (t^{-1}(\mathcal{F}_0))^{\bot} \subset \mathcal{F}_1 \subset  t^{-1}(\mathcal{F}_0).
\]
We can see that these $ \{\mathcal{F}_1\}$ correspond to Lagrangian subspaces in $t^{-1}(\mathcal{F}_0)/ (t^{-1}(\mathcal{F}_0))^{\bot}$ which have dimension $2s$. This is a smooth $s(s+1)/2$-dimensional scheme which we denote by $L(s,2s)$. From the above discussion we see that $\pi^{-1}(\mathcal{Q}(s,n)) $ is a $L(s,2s) $-bundle over $ \mathcal{Q}(s,n)$ with dimension $rs$. Thus, $\pi^{-1}(\mathcal{Q}(s,n)) $ is an irreducible component of $\mathcal{M}\otimes k $ which intersects with $\tau^{-1} (t\Lambda)$ over $ \mathcal{Q}(s,n) $.

In \cite{Kr}, Kr{\"a}mer shows that $\tau$ defines a resolution of  ${\rm M}^{\wedge}$ in the case that the signature is $ (n-1,1)$. In particular, she proves that $\mathcal{M}$ is regular with special fiber a reduced divisor with simple normal crossings. Also she shows that the special fiber of $\mathcal{M}$ consists of two smooth irreducible components of dimension $n - 1$ --- one of which being isomorphic to $\mathbb{P}^{n-1}_k$ (this corresponds to $\tau^{-1} (t\Lambda)$), and the other one being
a $\mathbb{P}^{1}_k$-bundle over a smooth quadric (this corresponds to $\pi^{-1}(\mathcal{Q}(1,n))$)--- which intersect transversely in a smooth irreducible variety of dimension $n -2$ (this corresponds to $\mathcal{Q}(1,n)$).

\section{An affine chart}\label{AffineChart}

The goal of this section is to write down the equations that define $\mathcal{M}$ in a neighbourhood $\U$ of $(\mathcal{F}_0, t \Lambda)$ for a general signature $(r,s)$; see Proposition \ref{redeq}. We also deduce, see Proposition \ref{2G_translates}, that $\mathcal{G}$-translates of $\mathcal{U} $ cover $\mathcal{M}$. (Recall from \S \ref{Prelim} that $(\mathcal{F}_0, t \Lambda) $ is a point in the fiber of $\tau : \mathcal{M} \rightarrow {\rm M}^{\wedge}\otimes_{O} O_F $ over the worst point.) In order to find the explicit equations that describe $\U$, we use similar arguments as in the proof of  \cite[Theorem 4.5]{Kr}.  In our case we consider: 
\[  
   \mathcal{F}_1 =    \left[\ 
\begin{matrix}[c]
A\\ \hline
I_n  
\end{matrix}\ \right], \quad \mathcal{F}_0 = X =  \left[\ 
\begin{matrix}[c]
X_1\\ \hline
X_2  
\end{matrix}\ \right] 
\]
where $A$ is of size $ n \times n$, $X$ is of size $2n \times s$ and $X_1,X_2$ are of size $n\times s$; with the additional condition that $\mathcal{F}_0$ has rank $s$ and so $X$ has a non-vanishing $s\times s$-minor. We also ask that $(\mathcal{F}_0,\mathcal{F}_1)$ satisfy the following four conditions:
\begin{enumerate}
            \begin{multicols}{2}
            \item $\mathcal{F}_1^{\bot} = \mathcal{F}_1$,  
            \item $ \mathcal{F}_0 \subset \mathcal{F}_1$,

            \columnbreak
            \item $(t - \pi)\mathcal{F}_0 = (0)$,
            
            \item $(t + \pi) \mathcal{F}_1 \subset \mathcal{F}_0$.

            \end{multicols}
        \end{enumerate}
Observe that
\[
M_t = \left[\ 
\begin{matrix}[c|c]
0_n & pI_n  \\ \hline
I_n & 0_n 
\end{matrix}\ \right]
\]
of size $2n \times 2n $ is the matrix giving multiplication by $t$.       
\begin{enumerate}
    \item The condition that $\mathcal{F}_1$ is isotropic translates to
    \[
        A^{t} = A .
    \]
    \item The condition $ \mathcal{F}_0 \subset \mathcal{F}_1$ translates to
\[
 \exists \, Y : X =  \left[\ 
\begin{matrix}[c]
A\\ \hline
I_n  
\end{matrix}\ \right] \cdot Y 
\]
where $Y$ is of size $ n\times s$. Thus, we have 
\[
\left[\ 
\begin{matrix}[c]
X_1\\ \hline
X_2  
\end{matrix}\ \right]  = \left[\ 
\begin{matrix}[c]
A\\ \hline
I_n  
\end{matrix}\ \right] \cdot Y  = \left[\ 
\begin{matrix}[c]
A   Y\\ \hline
I_n     Y
\end{matrix}\ \right] 
\]
and so $ X_1 = AY$ and $X_2 = Y.$ 
\item The condition $(t - \pi)\mathcal{F}_0 = (0)$ is equivalent to 
    \[
       M_t \cdot X =   \left[\ 
\begin{matrix}[c]
\pi X_1\\ \hline
\pi X_2
\end{matrix}\ \right],
    \]  
which amounts to 
\[
  \left[\ 
\begin{matrix}[c]
p X_2\\ \hline
 X_1
\end{matrix}\ \right]= \left[\ 
\begin{matrix}[c]
\pi X_1\\ \hline
\pi X_2
\end{matrix}\ \right].
\]
Thus, $X_1 = \pi X_2$ which translates to $A Y = \pi Y $ by condition $(2)$.  
\item The last condition $(t + \pi) \mathcal{F}_1 \subset \mathcal{F}_0$ translates to 
\[
\exists \, Z  :  M_t \cdot \left[\ 
\begin{matrix}[c]
A   \\ \hline
I_n     
\end{matrix}\ \right]  + \left[\ 
\begin{matrix}[c]
\pi A   \\ \hline
\pi I_n     Y
\end{matrix}\ \right]  = X \cdot Z^t
\]
where $Z$ is of size $n\times s$. This amounts to 
\[
  \left[\ 
\begin{matrix}[c]
p I_n + \pi A    \\ \hline
A + \pi I_n  
\end{matrix}\ \right] =  \left[\ 
\begin{matrix}[c]
X_1 Z^t   \\ \hline
X_2 Z^t
\end{matrix}\ \right] .
\]
From the above we get $A + \pi I_n  = X_2 Z^t$ which by condition (2) translates to $ A = Y Z^t-\pi I_n. $ Thus, $A$ can be expressed in terms of $Y,Z$. In addition, by condition (2) and in particular by the relations $ X_1 = AY$ and $X_2 = Y$ we deduce that the matrix $X$ is given in terms of $Y,Z$. Also from $Y = X_2$ we get that the matrix $Y$ is given in terms of $A,X$. (Below we will also show that $Z$ can be expressed in terms of $A,X$.)  
\end{enumerate}

For later use, we break up the matrices $Y,Z$ into blocks as follows. We write
\[
Y = \left[\ 
\begin{matrix}[c]
Y_1\\ \hline
Y_2  
\end{matrix}\ \right], \quad Z = \left[\ 
\begin{matrix}[c]
Z_1\\ \hline
Z_2  
\end{matrix}\ \right]
\]
where $Y_1,Z_1 $ are of size $s\times s$ and $Y_2, Z_2$ are of size $(n-s) \times s $. Observe from $X_1 = \pi X_2$ that all the entries of $X_1$ are in the maximal ideal and thus a minor involving entries of $X_1$ cannot be a unit. Recall that the matrix $X$ has a non-vanishing $s\times s$-minor and $X_2=Y$ from condition (2). Therefore, we can assume that $Y_1= I_s$ up to a change of basis.

We replace $A$ by $Y Z^t-\pi I_n$. Hence, conditions ($1$) and ($3$) are equivalent to
\begin{eqnarray}
    Z Y^t &=& Y Z^t, \\
    Y Z^t Y &=& 2 \pi Y. 
\end{eqnarray}
Here, we want to mention how we can express $Z$ in terms of $A,X$. From the above we have $Y Z^t=A+ \pi I_n$ and $ Y = \left[\ 
\begin{matrix}[c]
I_s\\ \hline
Y_2  
\end{matrix}\ \right]$ which give $ \left[\ 
\begin{matrix}[c]
Z^t\\ \hline
Y_2 Z^t  
\end{matrix}\ \right] = A+ \pi I_n$. Next, we break the matrices $ A, I_n$ into blocks: $A = \left[\ 
\begin{matrix}[c]
A_1\\ \hline
A_2  
\end{matrix}\ \right], \, I_n = \left[\ 
\begin{matrix}[c]
I_1\\ \hline
I_2  
\end{matrix}\ \right]$ where $A_1,I_1 $ are of size $s\times n$ and $A_2, I_2$ are of size $(n-s) \times n $. Hence, from $ \left[\ 
\begin{matrix}[c]
Z^t\\ \hline
Y_2 Z^t  
\end{matrix}\ \right] = A+ \pi I_n$ we obtain $Z^t = A_1 + \pi I_1$ and thus $ Z = A^t_1 + \pi I^t_1.$

From the above we deduce that an affine neighborhood of $\mathcal{M}$ around $(\mathcal{F}_0, t \Lambda)$ is given by $ \U = \Spec O_F[Y , Z ]/(Y_1 - I_s, \, Z Y^t - Y Z^t,\, Y Z^t Y - 2 \pi Y)$.
    
Our goal in this section is to prove the simplification of
equations given by the following proposition.
\begin{Proposition}\label{redeq}
The affine chart $\U \subset \mathcal{M}$ is isomorphic to 
\[
\Spec  O_F[Y_2 , Z_1 ]/(Z_1-Z^t_1,\, Z_1(I_s + Y^t_2 Y_2)-2\pi I_s ).
\]
\end{Proposition}
\begin{proof}
From (3.0.1) we get 
\[
\left[\ 
\begin{matrix}[c]
Z_1\\ \hline
Z_2  
\end{matrix}\ \right] \cdot [I_s \, | \, Y^t_2 ] = \left[\ 
\begin{matrix}[c]
I_s\\ \hline
Y^t_2  
\end{matrix}\ \right] \cdot [Z^t_1\, | \, Z^t_2],
\]
which gives 
\[
\left[\ 
\begin{matrix}[c|c]
Z_1 & Z_1 Y^t_2  \\ \hline
Z_2 & Z_2 Y^t_2 
\end{matrix}\ \right] = \left[\ 
\begin{matrix}[c|c]
Z^t_1 & Z^t_2  \\ \hline
Y_2 Z_1 & Y_2 Z^t_2 
\end{matrix}\ \right].
\]
From the above relation we obtain the relations $Z_1 = Z^t_1$ and $Z_2 = Y_2 Z^t_1.$ Thus, $Z_1$ is symmetric and $Z_2$ can be expressed in terms of $Y_2,Z_1.$ 

Next, the relation $(3.0.2)$ amounts to
\[
\left[\ 
\begin{matrix}[c]
I_s \\ \hline
Y_2  
\end{matrix}\ \right] \cdot [Z^t_1 \, | \, Z^t_2 ] \cdot \left[\ 
\begin{matrix}[c]
I_s \\ \hline
Y_2  
\end{matrix}\ \right] = \left[\ 
\begin{matrix}[c]
2\pi I_s \\ \hline
2 \pi Y_2  
\end{matrix}\ \right]
\]
which is equivalent to 
\[
\left[\ 
\begin{matrix}[c]
Z^t_1 + Z^t_2 Y_2 \\ \hline
Y_2 Z^t_1 + Y_2 Z^t_2 Y_2  
\end{matrix}\ \right] = \left[\ 
\begin{matrix}[c]
2\pi I_s \\ \hline
2 \pi Y_2  
\end{matrix}\ \right].
\]
From this we get $Z^t_1 + Z^t_2 Y_2 = 2 \pi I_s $ which translates to $ Z_1(I_s + Y^t_2 Y_2) = 2 \pi I_s$ as $Z^t_1 = Z_1$ and $Z_2 = Y_2 Z^t_1.$ From all the above the proof of the proposition follows.
\end{proof}
As corollaries of the above result we have:
\begin{Corollary}
For $(r,s) = (n-1,1)$, the corresponding affine chart $\U$ will be isomorphic to: 
\[
\U \cong \Spec (   O_F[ (y_{i})_{1 \leq i \leq n }, a ]/(a\sum^{n}_{c=1} y^2_{c}-2\pi,\, y_{1}-1)).
\]\qed
\end{Corollary}
\begin{Remark}\label{KramerReg}
{\rm For the above signature, Kr{\"a}mer \cite{Kr} shows that $\U$ is regular with special fiber a divisor with simple normal crossings.}  
\end{Remark}
\begin{Corollary}\label{n_2}
For $(r,s) = (n-2,2)$ the corresponding affine chart $\U$ will be isomorphic to: 
\[
\U \cong \Spec(   O_F[ (x_i)_{3\leq i \leq n}, (y_i)_{3\leq i \leq n}, a,b, c ]/ ( Z_1  N - 2\pi I_2))
\]
where 
\begin{equation*}
Z_1 = \begin{pmatrix} 
         a &  b  \\        
          b  & c  \\
         
    \end{pmatrix}, \quad N =\begin{pmatrix} 
    
      1+\sum^n_{i=3} x^2_i   &   \sum^n_{i=3} x_i y_i  \\
          \sum^n_{i=3} x_i y_i   &  1+\sum^n_{i=3} y^2_i    \\
    \end{pmatrix}. \qed 
\end{equation*}
\end{Corollary}
\begin{Remark}\label{Gen.Sign.}
{\rm In this case $s=2$, $\U$ does not have semi-stable reduction as one of the irreducible components of the special fiber of $\U$ is not smooth. More precisely, over the special fiber ($\pi =0$) $\U$ has three irreducible components given by 
\[
T_i = \{(Z_1,N)\,| \, Z_1N= 0,\, \text{rank}\,  Z_1 \leq i,\, \text{rank} \, N \leq 2-i\}, \quad \text{for} \quad i = 0,1,2.
\]
We can easily see that $T_0, T_2$ are smooth but $T_1$ is singular. We resolve the singularities of $\U$ in \S 3 by blowing up the irreducible component $T_0$ or in other words by blowing up the ideal $ (Z_1)$ generated by the entries of $Z_1$. Observe from the above and the proof of Proposition \ref{redeq} that $ A= Y \cdot \left[\ 
\begin{matrix}[c]
Z_1\\ \hline
Z_2  
\end{matrix}\ \right]^t $ and $Z_2 = Y_2 Z^t_1$ over the special fiber. Hence, $Z_1 = 0$, i.e. $ a =b=c=0$, gives $A=0$ which corresponds to $\mathcal{F}_1 = t \Lambda$. Thus $T_0 =  \U \cap \tau^{-1}(t\Lambda)$ where $\tau : \mathcal{M} \rightarrow {\rm M}^{\wedge}\otimes_{O} O_F $  (see \S 2.1). 
} 
\end{Remark}
\begin{Remark}\label{F0_In_U}
 {\rm
For general signature $(r,s)$, over the special fiber of $\U$ we have $Z_1=Z^t_1$ and $Z_1(I_s + Y^t_2 Y_2)=0.$ As in Remark \ref{Gen.Sign.}, $Z_1 =0$ gives $A=0$ which corresponds to $\mathcal{F}_1 = t \Lambda$.

Moreover, from the above and the definition of the (non-degenerate) symmetric form $\{\,,\,\}$ on $ t \Lambda\otimes k $ (see \S 2.1) we have $ \{\mathcal{F}_0,\mathcal{F}_0\} =Y^t\cdot Y =I_s + Y^t_2 Y_2  $ since $\mathcal{F}_0 =  \left[\  \begin{matrix}[c] 
X_1\\ \hline
X_2  
\end{matrix}\ \right]  $ where $X_1 = \pi X_2$, $ X_2 = Y$ and  $ Y = \left[\ 
\begin{matrix}[c]
I_s\\ \hline
Y_2  
\end{matrix}\ \right]$.
Thus, from the rank of $ I_s + Y^t_2 Y_2$ we read how isotropic $\mathcal{F}_0$ is with respect to $\{\,,\,\}$. When the rank of the matrix $I_s + Y^t_2 Y_2$ is zero, which actually occurs, we have $ \{\mathcal{F}_0,\mathcal{F}_0\}=0$. 

From all the above, we can easily see that $\U$ contains points  $(\mathcal{F}_0,\mathcal{F}_1)$ where $\mathcal{F}_0 \in \mathcal{Q}(s,n)$ and $\mathcal{F}_1=t\Lambda$. (Recall from \S 2.1 that $\mathcal{Q}(s,n)$ is the closed subscheme of $Gr(s,n)\otimes k $ which contains all the totally isotropic $s$-subspaces $\mathcal{F}_0$ with respect to the symmetric form $\{\,,\,\}$.)
 }
\end{Remark}
\begin{Proposition}\label{2G_translates}
When $s\geq 1,$ $\mathcal{G}$-translates of $\U$ cover $\mathcal{M}$.
\end{Proposition}
\begin{proof}
From \S 2.1, we have the forgetful $\calG$-equivariant morphism $\tau : \mathcal{M} \rightarrow {\rm M}^{\wedge}\otimes_{O} O_F$ given by $(\mathcal{F}_0,\mathcal{F}_1) \mapsto \mathcal{F}_1$. As in \cite[\S 4]{P} and \cite[\S 2.4.2, 5.5]{PR}, the worst
point of ${\rm M}^{\wedge}\otimes_{O} O_F$ is given by the $k$-valued point $t\Lambda$. The reason for this
terminology is that the geometric special fiber of ${\rm M}^{\wedge}\otimes_{O} O_F$ embeds into an appropriate affine flag variety, where it decomposes into unions
of finitely many Schubert cells, and the worst point is the unique closed Schubert
cell. This one point stratum lies in the closure of any other stratum and the
inclusion relations between the Schubert varieties are given by the Bruhat order. From the construction of splitting local models and the above, in order to show that $\mathcal{G}$-translates of $\U$ cover $\mathcal{M}$ it is enough to prove that $\mathcal{G}$-translates of $\U$ cover $\tau^{-1}(t\Lambda)$.  

Recall that $K$ is the stabilizer of $\Lambda$ in $G(\mathbb{Q}_p)$ and $ K^\circ$ is the neutral component of $K$, i.e. the parahoric stabilizer of $\Lambda$. When $n$ is odd $K = K^\circ$ and when $n$ is even $K/ K^\circ \simeq \mathbb{Z}/2\mathbb{Z} $; see \S2. Also, $\calG$ is the smooth group scheme of automorphisms of the polarized chain $\mathcal{L}$ over $\mathbb{Z}_p$ with $\mathcal{G}(\mathbb{Z}_p)=K$ and $\calG^\circ$ is the neutral component of $\calG$ with $\mathcal{G}^\circ(\mathbb{Z}_p)=K^\circ$. 

From \S 2.1, we have that $ \tau^{-1}(t\Lambda) \cong Gr(s,n)\otimes k$ and $\mathcal{G}_k$ acts via its action by reduction to $t\Lambda/t^2\Lambda$. This action factors through the orthogonal group $O(n)_k$ of the symmetric form $\{\,,\,\}$ on $ t \Lambda\otimes k $ and gives the map $ \mathcal{G}_k \rightarrow O(n)_k$. As in \cite[\S 4]{PR3} (see also \cite[\S 3.11]{Tits}), $\calG^\circ_k$ has $SO(n)_k$ as its maximal reductive quotient if $n$ is even and $O(n)_k$ if $n$ is odd via its action by reduction to $t\Lambda/t^2\Lambda$. The maps $ \calG^\circ_k \rightarrow O(n)_k$ and $ \calG^\circ_k \rightarrow SO(n)_k$ are surjective when $n$ is odd and even respectively. Therefore, the map $\mathcal{G}_k \rightarrow O(n)_k$ is always surjective. 

Next, the $O(n)$-action on $Gr(s, n)$ has a finite number of orbits. More precisely, there are $s+1$ orbits $Q(0),Q(1), \cdots,Q(s) $ where $Q(i) = \{ \mathcal{F}_0 \in Gr(s, n) | \text{dim}(rad(\mathcal{F}_0)) \\ = i \}$ (see  \cite[\S 4]{BDE}). For example in the case $s=2$ there are three $O(n)$-orbits: $\mathcal{F}_0$ can either contain no isotropic vectors at all or one isotropic vector or be totally isotropic. Observe that $Q(j)$ is contained in the (Zariski) closure of $Q(i)$ if and only if $ j \geq i$ and $Q(s) = \mathcal{Q}(s,n)$ is the unique closed orbit (see for example \cite[\S 3.1]{BDE} and \cite{ACGH}). Thus, $\mathcal{Q}(s,n)$ is contained in the closure of each orbit $Q(i).$

Lastly, from Remark \ref{F0_In_U} we have that $\U$ contains points $(\mathcal{F}_0,t\Lambda)$ with $\mathcal{F}_0 \in \mathcal{Q}(s,n) $ and so $\U$ contains points from all the orbits. Therefore, from all the above we deduce that the $\mathcal{G}$-translates of $\U$ cover $\tau^{-1}(t\Lambda)$. 
\end{proof}

\begin{Conjecture}
When $s \geq 1$, the scheme $\mathcal{M}$ is flat over $\Spec O_F$, Cohen-Macaulay and normal. It's special fiber is reduced.
\end{Conjecture}
\begin{Remark}
{\rm 
a) By Proposition \ref{2G_translates}, to prove the conjecture, it is enough to show that the affine chart $\mathcal{U}$ has the above properties. More precisely, the hard part of the conjecture is to prove that the special fiber of $ \mathcal{U}$ is reduced and Cohen-Macaulay.

b) For $(r,s) = (n-1,1)$, the conjecture is true as we can see from Remark \ref{KramerReg}.

c) The above conjecture is supported by some computer calculations that we obtained with the help of Macaulay 2. In particular, we verified the conjecture when $(r,s)= (n-2,2)$ where $n = 5,6,7,8,9,10$ for various primes $p > 2$. 
}
\end{Remark}

\section{A Blow-up}\label{BlowUp}
In what follows, we assume $(r,s) = (n-2,2).$ The goal of this section is to find a semi-stable resolution of the affine chart $\U$ (see Corollary \ref{SemStabCo}).

From Corollary \ref{n_2} we have that $\U \cong \Spec B $ where $B$ is the quotient ring 
\[
B = O_F[(x_i)_{3\leq i \leq n}, (y_i)_{3\leq i \leq n},a,b,c]/J
\]
and $ J$ is the ideal generated by the entries of the relation:
\[
\begin{pmatrix} 
         a &  b \\
           b &   c  \\
         
    \end{pmatrix}\begin{pmatrix} 
          Q(\und x)  &    P(\und x, \und y)  \\
            P(\und x, \und y) &  Q(\und y)  \\
         
    \end{pmatrix}=2 \pi \begin{pmatrix} 
         1 & 0  \\
          0 &  1   \\
         
    \end{pmatrix},
\]
where $ \displaystyle Q(\und x) = 1+\sum^n_{i=3} x^2_i $, $\displaystyle Q(\und y) = 1+\sum^n_{i=3} y^2_i $ and $\displaystyle  P(\und x, \und y)= \sum^n_{i=3} x_i y_i.$

Now, we will explicitly calculate the blow-up $\U^{\rm bl}$ of $\text{Spec}(B)$ along the ideal $(a,b,c)$. By Remark \ref{Gen.Sign.}, $\U^{\rm bl}$ is the blow-up of $\U$ along the smooth subscheme $ \U \cap \tau^{-1}(t\Lambda)$. Let $\rho: \U^{\rm bl} \rightarrow \U$ be the blow-up morphism. Define
\[
\U': = \text{Proj}\bigg(  B[t_1,t_2,t_3]/J' \bigg)
\]
where
$$
J' = \bigg( t_1  Q(\und x) - t_3 Q( \und y),  \,\,  t_2  Q(\und y) +  t_1 P(\und x, \und y),\,\, t_2 Q(\und x) + t_3 P(\und x, \und y), 
$$
$$
  at_2 - bt_1, \,\, at_3 - c t_1, \,\, bt_3 - ct_2 \bigg). 
$$
By definition, $\U^{\rm bl}$ is a closed subscheme of the projective $\text{Spec}(B)$-scheme $\U' $ (as the blow-up $\U^{\rm bl}$ may be, a priori, cut out by more equations). In fact, as a result of our
analysis we will see that $\U^{\rm bl} = \U'$.

\begin{Proposition}\label{SemStab}
a) $\U'$ has semi-stable reduction over $O_F$.

b) The closed immersion $\U^{\rm bl} \rightarrow \U'$ is an isomorphism.
\end{Proposition}

\begin{proof}

There are three affine patches that cover $ \U'$:\\
For $t_1 = 1$ the affine open chart is given by $V(J_1) = \text{Spec}R_1/J_1   
$ where $ R_1 = O_F[(x_i)_{3\leq i \leq n}, (y_i)_{3\leq i \leq n},a,t_2,t_3] $ and 
$$
J_1 = \left( t_2  Q(\und y) +  P(\und x, \und y), \,\,   Q(\und x) - t_3 Q(\und y), \,\, a(  Q(\und x) + t_2P(\und x, \und y)) - 2\pi  \right).
$$
We will show that the scheme $V(J_1)$ has semi-stable reduction over $O_F$. It suffices to prove that $V(J_1)$ is regular and its special fiber is reduced with smooth irreducible components that have smooth intersections with correct dimensions. First we observe that 
$$
J_1 = \left( t_2  Q(\und y) +  P(\und x, \und y), \,\,   Q(\und x) - t_3 Q(\und y), \,\, a( t_3-t^2_2) Q(\und y) - 2\pi  \right).
$$
Over the special fiber ($\pi =0$) we have $ V(\bar{J}_1) = \text{Spec}\bar{R}_1/ \bar{J}_1$ where 
\[ 
\bar{J}_1 = \left ( t_2  Q(\und y) + P(\und x, \und y), \quad Q(\und x) - t_3 Q(\und y) , \quad 
a( t_3-t^2_2) Q(\und y) \right)
\]
and $\bar{R}_1 = k[(x_i)_{3\leq i \leq n}, (y_i)_{3\leq i \leq n},a,t_2,t_3].$ Let $V(I_i) = \text{Spec} \bar{R}_1 /I_i $ of dimension $2(n-2)$, where
\begin{eqnarray*}
  I_1 &=& (a, \quad t_2  Q(\und y) +  P(\und x, \und y), \quad  Q(\und x) - t_3 Q(\und y)   ) \\
  I_2 &=& (t_3-t^2_2,  \quad t_2  Q(\und y) +  P(\und x, \und y), \quad  Q(\und x) - t^2_2 Q(\und y) ) \\
  I_3 &=& ( Q(\und y) , \quad   P(\und x, \und y), \quad  Q(\und x) ).
\end{eqnarray*}
We can easily see that 
\[
V(\bar{J}_1) = V(I_1) \cup V(I_2) \cup V(I_3). 
\]
Using the Jacobi criterion we see that $V(I_i)$ are smooth and that their intersections:
\begin{eqnarray*}
  I_1 + I_2  &=& (a, \quad t_3-t^2_2,  \quad t_2  Q(\und y) +  P(\und x, \und y), \quad  Q(\und x) - t^2_2 Q(\und y) ) \\
  I_1 + I_3 &=& (a,\quad  Q(\und y) , \quad   P(\und x, \und y), \quad  Q(\und x)   ) \\
  I_2 + I_3 &=& ( t_3-t^2_2, \quad  Q(\und y) , \quad   P(\und x, \und y), \quad  Q(\und x) )\\
   I_1 + I_2 + I_3 &=& (a,\quad  t_3-t^2_2, \quad  Q(\und y) , \quad   P(\und x, \und y), \quad  Q(\und x)   )
\end{eqnarray*}
are also smooth and with the correct dimensions. By the above, we get that $V(I_i)$ are the smooth irreducible components of $V(\bar{J}_1)$. 

Now, we prove that the special fiber of $ V(J_1)$ is reduced by showing that 
\[
\bar{J}_1 = I_1 \cap I_2 \cap I_3.
\]
Recall that $\bar{J}_1 = (m_1,m_2, a (t_3-t^2_2) Q(\und y))$ where $m_1 :=t_2  Q(\und y) +  P(\und x, \und y)$ and $m_2 := Q(\und x) - t_3 Q(\und y)$. Clearly, $\bar{J}_1 \subset I_1 \cap I_2 \cap I_3$. Let $g \in I_1 \cap I_2 \cap I_3.$ Thus, $g \in I_1$ and
\[
g = f_1 a + f_2 m_1 +f_3 m_2 \equiv  f_1 a \, \text{mod} \, \bar{J}_1
\]
for $f_i \in \bar{R}_1$. Also, $g \in I_2$ and so $f_1 a \in I_2$. $I_2$ is a prime ideal and $a \notin I_2$. Thus, $f_1 \in I_2$ and 
\[
f_1 = h_1(t_3-t^2_2)+h_2m_1+h_3m_2 \equiv h_1(t_3-t^2_2)  \, \text{mod} \, \bar{J}_1
\]
for $h_i \in \bar{R}_1$. Lastly, $g \in I_3$ and from the above we obtain $ h_1\in I_3$ as $a, (t_3-t^2_2) \notin I_3.$ Thus, 
\[
h_1 = k_1 Q(\und y) +k_2  P(\und x, \und y) +k_3 Q(\und x) \equiv Q(\und y)(k_1  -k_2 t_2  + k_3 t_3)  \, \text{mod} \, \bar{J}_1
\] 
for $k_i \in \bar{R}_1$. Therefore, $ g \equiv a (t_3-t^2_2) Q(\und y)(k_1  -k_2 t_2  + k_3 t_3) \equiv 0  \, \text{mod} \, \bar{J}_1 $ and so $g \in \bar{J}_1.$ Hence, $\bar{J}_1 = I_1 \cap I_2 \cap I_3.$

Next, we can easily see that the ideals $ I_1,I_2,I_3$ are principal over $V(J_1)$. In particular, for $I_1$ we have $I_1 = (a)$, for $I_2$ we have $I_2 = (t_3-t^2)$ and for $I_3$ we get $I_3 = (Q(\und y))$. From the above we deduce that $V(J_1)$ is regular (see \cite[Remark 1.1.1]{Ha}).

From all the above discussion we deduce that $V(J_1)$ has semi-stable reduction over $O$. By symmetry, we get similar results for $t_3=1$.\\
$\quad$\\
For $t_2=1$, the affine open chart is given by $V(J_2) = \text{Spec}R_2/J_2$ where $ R_2  = O_F[(x_i)_{3\leq i \leq n}, (y_i)_{3\leq i \leq n},b,t_1,t_3] $ and
$$
J_2 = \left(   Q(\und y) +  t_1 P(\und x, \und y), \,\,   Q(\und x) + t_3 P(\und x, \und y), \,\, b( 1- t_1t_3)P(\und x, \und y)- 2\pi  \right).
$$
To show that $V(J_2) $ has semi-stable reduction one proceeds exactly as above. In this case, the special fiber of $V(J_2)$ is isomorphic to $  \text{Spec}\bar{R}_2/ \bar{J}_2$ where
\[
\bar{J}_2 = \left ( Q(\und y) + t_1P(\und x, \und y), \quad
 Q(\und x) + t_3P(\und x, \und y) , \quad  
b( 1- t_1t_3)P(\und x, \und y) \right)
\]
and $\bar{R}_2  = k[(x_i)_{3\leq i \leq n}, (y_i)_{3\leq i \leq n},b,t_1,t_3].$ Let $V(I'_i) = \text{Spec} \bar{R}_2/I'_i  $ of dimension $2(n-2)$, where
\begin{eqnarray*}
  I'_1 &=& (b, \quad  Q(\und y) + t_1 P(\und x, \und y), \quad  Q(\und x) + t_3P(\und x, \und y)  ) \\
  I'_2 &=& (1- t_1t_3,  \quad  Q(\und y) + t_1 P(\und x, \und y), \quad  Q(\und x) + t_3P(\und x, \und y)) \\
  I'_3 &=& (  P(\und x, \und y) , \quad Q(\und y) , \quad  Q(\und x) ).
\end{eqnarray*} 
and their intersections 
\begin{eqnarray*}
  I'_1 + I'_2  &=& (b, \quad 1- t_1t_3,  \quad  Q(\und y) + t_1 P(\und x, \und y), \quad  Q(\und x) + t_3 P(\und x, \und y) ) \\
  I'_1 + I'_3 &=& (b,\quad   Q(\und y) , \quad   P(\und x, \und y), \quad  Q(\und x)   ) \\
  I'_2 + I'_3 &=& ( 1- t_1t_3, \quad  Q(\und y) , \quad   P(\und x, \und y), \quad  Q(\und x) )\\
  I'_1 + I'_2 + I'_3 &=& (b, \quad  1- t_1t_3,\quad  Q(\und y) , \quad   P(\und x, \und y), \quad  Q(\und x)   ).
\end{eqnarray*}
As in the case $t_1 =1$, by using the Jacobi criterion we see that the irreducible components $V(I'_i)$ are smooth and they intersect transversely. Also, by a similar argument as above we can easily see that $ V(J_2) $ is regular and its special fiber is reduced. Now, the semi-stability of $V(J_2)$ follows. 

By the above, we conclude that $\U'$ is regular, of relative dimension $2(n-2)$, that $\U' $ is $O_F$-flat and that its special fiber is a reduced divisor with normal crossings. This shows part (a). Let us show part (b). The blow-up $ \U^{\rm bl}$ is a closed subscheme of $\U'$. By the above, $ \U'$ is integral of
dimension $ 2(n-2)$. However, the dimension of the blow-up $\U^{\rm bl} $ is also $2(n-2)$. Indeed, on one hand $\U^{\rm bl} $ is a closed subscheme of $\U'$ while on the other hand it is birational to $\text{Spec}(B)$. We deduce that $ \U^{\rm bl}= \U'$ which is the claim in (b).
\end{proof}
As a consequence of the above proposition we obtain:
\begin{Corollary}\label{SemStabCo}
The morphism $\rho: \U^{\rm bl} \rightarrow \U$ is a semistable resolution, i.e. $\U^{\rm bl} $ has semi-stable reduction over $O_F$.
\end{Corollary}
\begin{proof}
It follows from part (a) and (b) of Proposition \ref{SemStab}.
\end{proof}

\RC{The last affine chart that we have to check is when $\mathbf{t_4=1}$. We get $a = t_1 \pi$, $b = t_2 \pi $, $c=t_3 \pi$ and so the affine chart is given by:
$$
T_4 = \text{Spec}\bigg(O[\overline{x}, \overline{y},t_1,t_2,t_3]/J_4   \bigg)
$$
where
$$
J_4 = \bigg( \begin{pmatrix} 
        t_1 &  t_2 \\
           t_2 &   t_3  \\
         
    \end{pmatrix}\begin{pmatrix} 
          Q(\und x)_{i}  &    \sum' x_{i}y_{i}  \\
            \sum' x_{i}y_{i} &    Q(\und y)_{i}  \\
         
    \end{pmatrix}-2  I_2 \bigg).
 $$
Let $t := t_1 t_3-t^2_2 $ and $B_4 = O[\overline{x}, \overline{y},t_1,t_2,t_3]/J_4 $. The blow-up of $T_4$ along $(t)$ is equal with $ \text{Spec}(B_4/J_{\text{an}})$ where 
$$
J_{\text{an}} = \text{Ker}( B_4 \rightarrow (B_4)_t).
$$
We obtain that 
$$ 
B_4/J_{\text{an}} \cong O[\overline{x},\overline{y},t_1,t_2,t_3]/\bigg((t_1t_3 -t^2_2)\sum' x_iy_i+\frac{1}{2}t_2,\,(t_1t_3 -t^2_2) Q(\und x)-\frac{1}{2}t_3,\,(t_1t_3 -t^2_2) Q(\und y)-\frac{1}{2}t_1 \bigg).
$$
$\text{Spec}(B_4/J_{\text{an}} )$ is smooth and irreducible.
}
\begin{Remark}
{\rm From the proof of Proposition \ref{SemStab} we obtain that the special fiber of $\U^{\rm bl} $ has three irreducible components. In fact, we
explicitly describe the equations defining these irreducible components over the three affine patches that cover $\U^{\rm bl}$. It is then easy to see that the exceptional locus of $\rho: \U^{\rm bl} \rightarrow \U$ is the irreducible component of the special fiber of $\U^{\rm bl}$
\[
\text{Proj}\left(\frac{k[(x_i)_{3\leq i \leq n}, (y_i)_{3\leq i \leq n}][t_1,t_2,t_3]}{(t_1  Q(\und x) - t_3 Q( \und y),  \,  t_2  Q(\und y) +  t_1 P(\und x, \und y),\, t_2 Q(\und x) + t_3 P(\und x, \und y)) } \right)
\]
that corresponds to $V(I_1)$ and $V(I'_1)$ for the affine patches $t_1=1$ and $t_2=1$ respectively.}
\end{Remark}

\section{A resolution for the local model}\label{Resol}
We use the notation from \S \ref{Prelim}. In particular, recall the morphism 
\[\tau : \mathcal{M} \rightarrow {\rm M}^{\wedge}\otimes_O O_F\]
and the following isomorphisms over the generic fiber
\begin{equation}\label{GenFib}
\mathcal{M}\otimes F \cong  {\rm M}^{\wedge} \otimes F \cong  \Mloc \otimes F.
\end{equation}

Let $\mathcal{Z} = \tau^{-1}(t\Lambda)$ be the smooth $\mathcal{G}$-invariant subscheme of dimension $2(n-2)$, which is supported in the special fiber. (Recall from \S \ref{Prelim} that $t\Lambda$ is the worst point of ${\rm M}^{\wedge}$ and $ \tau^{-1}(t\Lambda) \cong Gr(2,n)\otimes k$.) We consider the blow-up of $\mathcal{M}$ along the subscheme $\mathcal{Z}$. This gives a $\mathcal{G}$-birational projective morphism 
\[
r^{\rm bl} : {\rm M}^{\rm bl} \rightarrow \mathcal{M}
\]
 which induces an isomorphism on the generic fibers. 
\begin{Theorem}\label{MSemi}
The scheme ${\rm M}^{\rm bl}$ is regular and has special fiber a reduced divisor with normal crossings.
\end{Theorem}
\begin{proof}
From Proposition \ref{2G_translates} we have that the $\mathcal{G}$-translates of $\U$ cover $\mathcal{M}$ and since $r^{\rm bl} $ is $\mathcal{G}$-equivariant we obtain that the $\mathcal{G}$-translates of the open $\U^{\rm bl} = (r^{\rm bl})^{-1}(\U) \subset {\rm M}^{\rm bl}$ cover ${\rm M}^{\rm bl}$. Therefore, it is enough to show the conclusion of the theorem for the blow-up $\U^{\rm bl}$ of $\U$ at the ideal $  (a , \,   b, \,   c)  $ and by Corollary \ref{SemStabCo} the proof of the theorem follows.
\end{proof}
\begin{Remark} {\rm It would be useful to have a simple moduli-theoretic description of the blow-up ${\rm M}^{\rm bl}$ similar in spirit to the description of $\mathcal{M}$ given in \S \ref{Prelim}.} \end{Remark}
 
We just proved that ${\rm M}^{\rm bl} $ has semi-stable reduction, and is therefore flat over $O_F$. Combining all the above we have: 

\[{\rm M}^{\rm bl} \xrightarrow{r^{\rm bl}} \mathcal{M} \xrightarrow{\tau} {\rm M}^{\wedge}\otimes_O O_F \]
which factors through $\Mloc \otimes_O O_F \subset {\rm M}^{\wedge} \otimes_O O_F$ because of flatness; the generic fiber of all of these is the same as we can see from (\ref{GenFib}).
Then, we obtain that  ${\rm M}^{\rm bl} \rightarrow \Mloc \otimes_O O_F$ is a $\mathcal{G} $-equivariant birational projective morphism.

\section{Application to Shimura varieties}\label{Shimura}
\subsection{Unitary Shimura data}
We now discuss some Shimura varieties to
which we can apply these results. We follow \cite[\S 1.1]{PR} for the description of the unitary Shimura varieties; see also \cite[\S 3]{P}. 

Let $F_0$ be an imaginary quadratic field and fix an embedding $\epsilon : F_0 \hookrightarrow \mathbb{C}$. Let $O$ be the ring of integers of $F_0$ and denote by $a \mapsto \overline{a}$ the non
trivial automorphism of $F_0$. Assuming $n>3$, we let $W = F^n_0$ be a $n$-dimensional $F_0$-vector space, and we suppose
that $\phi : W \times W \rightarrow F_0$ is a non-degenerate hermitian form. Set
$W_{\mathbb{C}}= W \otimes_{ F_0,\epsilon} \mathbb{C}$. Choosing a suitable isomorphism $W_{\mathbb{C}} \cong \mathbb{C}^n$ we may write $\phi$ on $ W_{\mathbb{C}}$ in a normal form $\phi (w_1,w_2) = \,  ^t  \bar{w}_1 H w_2$ where
\[
H = \text{diag}(-1, \dots , -1, 1, \dots, 1).
\]
We denote by $s$ (resp. $r$) the number of places, where $-1$, (resp. $1$) appears in $H$. We will say that $\phi$ has signature $(r,s)$. By replacing $\phi$ by $-\phi$ if needed, we can make sure that $s\leq r$ and so we assume that $s \leq r$. Let $J : W_{\mathbb{C}} \rightarrow W_{\mathbb{C}}$ be the
endomorphism given by the matrix $- \sqrt{-1}H$. We have $J^
2 = -id$ and so the endomorphism $J$ gives an $\mathbb{R}$-algebra homomorphism $h_0 : \mathbb{C} \rightarrow \text{End}_{\mathbb{R}}(W\otimes_{\mathbb{Q}} \mathbb{R})$ with $h_0(\sqrt{-1}) = J$ and
hence a complex structure on $W\otimes_{\mathbb{Q}} \mathbb{R} = W_{\mathbb{C}}$. For this complex structure we have
\[
\text{Tr}_{\mathbb{C}}(a; W \otimes_{\mathbb{Q}} \mathbb{R}) = s \cdot \epsilon (a) + r \cdot \bar{\epsilon}(a), \quad  a \in   F_0.
\]
Denote by $E$ the subfield of $\mathbb{C}$ which is generated by the traces above (the ``reflex field"). We have that $E = \mathbb{Q}$ if $r = s$ and $E=F_0$ otherwise. The representation of $F_0$ on $W \otimes_{\mathbb{Q}} \mathbb{R}$ with the
above trace is defined over $E$, i.e there is an $n$-dimensional $E$-vector space $W_0$ on which $F_0$ acts such that
\[
 \text{Tr}_E(a; W_0) = s \cdot a + r \cdot \bar{a}
\]
and such that $W_0 \otimes_E \mathbb{C}$ together with the above $ F_0$-action is isomorphic to $W \otimes_{\mathbb{Q}} \mathbb{R}$ with
the $ F_0$-action induced by $\epsilon :  F_0 \hookrightarrow \mathbb{C}$ and the above complex structure.

Next, fix a non-zero element $a \in  F_0$ with $ \bar{a}= -a$ and set
\[
\psi (x,y) = \text{Tr}_{ F_0/ \mathbb{Q}}(a^{-1}
\phi (x,y))
\]
which is a non-degenerate alternating form $W \otimes_{\mathbb{Q}} W \rightarrow \mathbb{Q}$. This satisfies
\[
\psi (av,w) = \psi  (v, \bar{a}w ), \quad \text{for all} \,\,\, a \in  F_0,\, v,w \in W.
\]
By replacing $a$ by $-a$, we can make sure that the symmetric $\mathbb{R}$-bilinear form on $W_{\mathbb{C}}$ given by $\psi(x,Jy)$ for $x, y \in W_{\mathbb{C}}$ is positive definite. Let $G$ be the reductive group over $\mathbb{Q}$ which is given by
\[
G(\mathbb{Q}) = \{g \in GL_{F_0} (W) \,|\, \psi(gv,gw) = c(g)\psi(v,w), \,  c(g) \in \mathbb{Q}^{\times}\} .
\]
The group $G$ can be identified with the unitary similitude group of the form $\phi$. Set
\[
GU(r,s) := \{A \in GL_n(\mathbb{C}) \, | \,
^t \bar{A}HA ¯ = c(A)H, c(A) \in \mathbb{R}^{\times}\}.
\]
By the above, the embedding $\epsilon :  F_0 \hookrightarrow \mathbb{C}$ induces an isomorphism $G(\mathbb{R}) \cong
GU(r,s)$.
We define a homomorphism $h : \text{Res}_{\mathbb{C}/\mathbb{R}}\mathbb{G}_{m,\mathbb{C}} \rightarrow G_{\mathbb{R}}$ by restricting $h_0$ to $\mathbb{C}^{\times}$. Then $h(a)$ for $a \in \mathbb{R}^{\times}$ acts on $W \otimes_{\mathbb{Q}} \mathbb{R}$ by multiplication by $a$ and $h(\sqrt{-1})$ acts as $J$. Consider $h_{\mathbb{C}}(z, 1) : \mathbb{C}^{\times} \rightarrow G(\mathbb{C}) \cong GL_n(\mathbb{C}) \times \mathbb{C}^{\times}$. Up to conjugation $h_{\mathbb{C}}(z, 1)$ is given by
\[
\mu_{r,s}(z) = (\text{diag}( z^{(s)},1^{(r)}), z);
\] 
this is a cocharacter of $G$ defined over the number field $E$. Denote by $X_h$ 
the conjugation orbit of $h(i)$ under $G(\mathbb{R})$. The pair $(G,h)$ gives rise to a Shimura variety $\text{Sh}(G,h)$ which is defined over the reflex field $E$.

\subsection{Unitary integral models}\label{un.int.}
We continue with the notations and assumptions of the previous paragraph. In particular, we take $G=GU_n$ and $X = X_h$ above that define the unitary similitude Shimura datum $(G, X)$. Assume that $(r,s) =(n-2,2).$

Assume that $p$ is an odd prime number and is ramified in $F_0$. Let $F_1 = F_0 \otimes \mathbb{Q}_p$ and $V = W \otimes_{\mathbb{Q}} \mathbb{Q}_p$. We fix a square root $\pi$ of $p$ and  we set $k = \overline{\mathbb{F}_p}$. In addition, we assume that the hermitian form $\phi$ on $ V$ is split. This means that there exists a basis $e_1,\dots , e_n$ of $ V$ such that $\phi(e_i, e_{n+1-j}) = \delta_{ij}$ for $ i, j \in \{1,\dots ,n\}.$ We denote by $\Lambda$ the standard lattice $O^n \otimes_{\mathbb{Z}} \mathbb{Z}_p $ in $V$. Denote by $K$ the stabilizer of $\Lambda$ in $G(\mathbb{Q}_p)$.

We let $\mathcal{L}$ be the
self-dual multichain consisting of $\{\pi^k \Lambda\}_{k \in \mathbb{Z}}$. Here $\mathcal{G} = \underline{{\rm Aut}}(\mathcal{L})$ is the group scheme over $\mathbb{Z}_p$ with $K = \mathcal{G}(\mathbb{Z}_p)$ the subgroup of $G(\mathbb{Q}_p)$ fixing the lattice chain $\mathcal{L}$. Denote by $K^\circ$ the neutral component of $K$. As in \S 2,  when $n$ is odd $K = K^\circ$ and when $n$ is even $K/ K^\circ \simeq \mathbb{Z}/2\mathbb{Z} $.

Choose also a sufficiently small compact open subgroup $K^p$ of the prime-to-$p$ finite adelic points $G({\mathbb A}_{f}^p)$ of $G$ and set $\mathbf{K}=K^pK$ and $\mathbf{K}'=K^pK^\circ $. As was observed in \cite[\S 1.3]{PR}, the Shimura varieties ${\rm Sh}_{\mathbf{K}'}(G, X)$ and ${\rm Sh}_{\mathbf{K}}(G, X)$ have isomorphic geometric connected components. Therefore, from the point of view of constructing reasonable integral models, we may restrict our attention to ${\rm Sh}_{\mathbf{K}}(G, X)$; since $K$ corresponds to a lattice set stabilizer, this Shimura variety is given by a simpler moduli problem.
 The Shimura variety  ${\rm Sh}_{\mathbf{K}}(G, X)$ with complex points
 \[
 {\rm Sh}_{\mathbf{K}}(G, X)(\mathbb{C})=G(\mathbb{Q})\backslash X\times G({\mathbb A}_{f})/\mathbf{K}
 \]
is of PEL type and has a canonical model over the reflex field $E$. We set $\mathcal{O} = O_{E_v}$ where $v$ the unique prime ideal of $E$ above $(p)$.

We consider the moduli functor $\mathcal{A}^{\rm naive}_{\mathbf{K}}$ over $\Spec \mathcal{O} $ given in \cite[Definition 6.9]{RZbook}:\\
A point of $\mathcal{A}^{\rm naive}_{\mathbf{K}}$ with values in the $\Spec \mathcal{O} $-scheme $S$ is the isomorphism class of the following set of data $(A,\bar{\lambda}, \bar{\eta})$:
\begin{enumerate}
    \item An $\mathcal{L}$-set of abelian varieties $A = \{A_{\Lambda}\}$.
    \item A $\mathbb{Q}$-homogeneous principal polarization $\bar{\lambda}$ of the $\mathcal{L}$-set $A$.
    \item A $K^p$-level structure
    \[
\bar{\eta} : H_1 (A, {\mathbb A}_{f}^p) \simeq W \otimes  {\mathbb A}_{f}^p \, \text{ mod} \, K^p
    \]
which respects the bilinear forms on both sides up to a constant in $({\mathbb A}_{f}^p)^{\times}$ (see loc. cit. for
details).

The set $A$ should satisfy the determinant condition (i) of loc. cit.
\end{enumerate}

For the definitions of the terms employed here we refer to loc.cit., 6.3–6.8 and \cite[\S 3]{P}. The functor $\mathcal{A}^{\rm naive}_{\mathbf{K}}$ is representable by a quasi-projective scheme over $\mathcal{O}$. Since the Hasse principle is satisfied for the unitary group, we can see as in loc. cit. that there is a natural isomorphism
\[
\mathcal{A}^{\rm naive}_{\mathbf{K}} \otimes_{\calO} E_v = {\rm Sh}_{\mathbf{K}}(G, X)\otimes_{E} E_v.
\]

As is explained in \cite{RZbook} and \cite{P} the naive local model ${\rm M}^{\rm naive}$ is connected to the moduli scheme $\mathcal{A}^{\rm naive}_{\mathbf{K}}$ via the local model diagram 
\[
\mathcal{A}^{\rm naive}_{\mathbf{K}} \ \xleftarrow{\psi_1} \Tilde{\mathcal{A}}^{\rm naive}_{\mathbf{K}} \xrightarrow{\psi_2} {\rm M}^{\rm naive}
\]
where the morphism $\psi_1$ is a $\mathcal{G}$-torsor and $\psi_2$ is a smooth and $\mathcal{G}$-equivariant morphism. Therefore, there is a relatively representable smooth
morphism
 \[
 \mathcal{A}^{\rm naive}_{\mathbf{K}} \to [\mathcal{G} \backslash  {\rm M}^{\rm naive}]
 \]
where the target is the quotient algebraic stack.

As we mentioned in \S \ref{Prelim}, the scheme ${\rm M}^{\rm naive}$ is never flat and by the above, the same is true for $\mathcal{A}^{\rm naive}_{\mathbf{K}}$. Denote by $ \mathcal{A}^{\rm flat}_{\mathbf{K}} $ the flat closure of ${\rm Sh}_{\mathbf{K}}(G, X)\otimes_{E} E_v$ in $ \mathcal{A}^{\rm naive}_{\mathbf{K}}$. Recall from \S \ref{Prelim} that the flat closure of $ {\rm M}^{\rm naive} \otimes_{\mathcal{O}} E_v$ in ${\rm M}^{\rm naive}$ is by definition the local model $\Mloc $. By the above we can see, as in \cite{PR}, that there is a relatively representable smooth morphism of relative dimension
${\rm dim} (G)$,
\[
\mathcal{A}^{\rm flat}_{\mathbf{K}} \to [\mathcal{G} \backslash \Mloc].
 \]
This of course implies imply that $\mathcal{A}^{\rm flat}_{\mathbf{K}}$ is \'etale locally isomorphic to the local model $\Mloc$.

One can now consider a variation of the moduli of abelian schemes $\mathcal{A}^{\rm spl}_{\mathbf{K}}$ where we add in the moduli problem an additional subspace in the Hodge filtration $ {\rm Fil}^0 (A) \subset H_{dR}^1(A)$ of the universal abelian variety $A$ (see \cite[\S 6.3]{H} for more details) with certain conditions to imitate the definition of the splitting local model $\mathcal{M}$. $\mathcal{A}^{\rm spl}_{\mathbf{K}}$ associates to an $O_{F_1}$-scheme $S$ the set of isomorphism classes of objects $(A,\bar{\lambda}, \bar{\eta},\mathscr{F}_0) $. Here $(A,\bar{\lambda}, \bar{\eta})$ is an object of  $\mathcal{A}^{\rm naive}_{\mathbf{K}}(S).$ Set $\mathscr{F}_1 := {\rm Fil}^0 (A) $. The final ingredient $\mathscr{F}_0$ of an object of $\mathcal{A}^{\rm spl}_{\mathbf{K}}$ is the subspace $ \mathscr{F}_0 \subset \mathscr{F}_1 \subset H_{dR}^1(A) $ of rank $s$ which satisfies the following conditions: 
\[
 (t + \pi) \mathscr{F}_1 \subset \mathscr{F}_0, \quad (t - \pi)\mathscr{F}_0 = (0).
\]
There is a forgetful morphism
\[
\tau :   \mathcal{A}^{\rm spl}_{\mathbf{K}} \longrightarrow \mathcal{A}^{\rm naive}_{\mathbf{K}}\otimes_{\mathcal{O}} O_{F_1}
\]
defined by $(A,\bar{\lambda}, \bar{\eta},\mathscr{F}_0) \mapsto (A,\bar{\lambda}, \bar{\eta}) $. Moreover, $\mathcal{A}^{\rm spl}_{\mathbf{K}}$ has the same \'etale local structure as $\mathcal{M}$; it is a ``linear modification" of $\mathcal{A}^{\rm naive}_{\mathbf{K}}\otimes_{\mathcal{O}} O_{F_1}$ in the sense of \cite[\S 2]{P} (see also \cite[\S 15]{PR2}). Also we want to mention that under the local model diagram the subspace $\mathscr{F}_1$ corresponds to $\mathcal{F}_1$ of $(\mathcal{F}_0, \mathcal{F}_1) \in \mathcal{M}$.

 \begin{Theorem}\label{RegLM}
     
 For every $K^p$ as above, there is a
 scheme $\mathcal{A}^{\rm bl}_{\mathbf{K}}$, flat over $\Spec(O_{F_1})$, 
 with
 \[
\mathcal{A}^{\rm bl}_{\mathbf{K}}\otimes_{O_{F_1}} F_1={\rm Sh}_{\mathbf{K}}(G, X)\otimes_{E} F_1,
 \]
 and which supports a local model diagram
  \begin{equation}\label{LMdiagramReg}
\begin{tikzcd}
&\wti{\mathcal{A}}^{\rm bl}_{\mathbf{K}}(G, X)\arrow[dl, "\pi^{\rm reg}_K"']\arrow[dr, "q^{\rm reg}_K"]  & \\
\mathcal{A}^{\rm bl}_{\mathbf{K}}  &&  {\rm M}^{\rm bl}
\end{tikzcd}
\end{equation}
such that:
\begin{itemize}
\item[a)] $\pi^{\rm reg}_{\mathbf{K}}$ is a $\mathcal{G}$-torsor for the parahoric group scheme $\calG$ that corresponds to $K_p$,

\item[b)] $q^{\rm reg}_{\mathbf{K}}$ is smooth and $\calG$-equivariant.

\item[c)] $\mathcal{A}^{\rm bl}_{\mathbf{K}}$ is regular and has special fiber which is a reduced divisor with
normal crossings. 
\end{itemize}
 \end{Theorem}
 \begin{proof}  
 By the above, we have
 \begin{equation}\label{LMdiagramLastChapter}
\begin{tikzcd}
&\wti{\mathcal{A}}^{\rm spl}_{\mathbf{K}}\arrow[dl, "\pi_{\mathbf{K}}"']\arrow[dr, "q_{\mathbf{K}}"]  & \\
\mathcal{A}^{\rm spl}_{\mathbf{K}}  &&  \mathcal{M}
\end{tikzcd}
\end{equation}
with $\pi_{\mathbf{K}}$ a $\mathcal{G}$-torsor and $q_{\mathbf{K}}$ smooth and $\mathcal{G}$-equivariant. We set
\[
\wti{\mathcal{A}}^{\rm bl}_{\mathbf{K}}=\wti{\mathcal{A}}^{\rm spl}_{\mathbf{K}}\times_{\mathcal{M}}{\rm M}^{\rm bl}
\]
which carries a diagonal $\mathcal{G}$-action. Since ${\rm M}^{\rm bl} \longrightarrow \mathcal{M}$ is given by a blow-up, is projective, and we can see  (\cite[\S 2]{P}) that the quotient
\[
\pi^{\rm reg}_K: \wti{\mathcal{A}}^{\rm bl}_{\mathbf{K}} \longrightarrow \mathcal{A}^{\rm bl}_{\mathbf{K}}:=\calG\backslash \wti{\mathcal{A} }^{\rm bl}_K(G, X)
\]
is represented by a scheme and gives a $\calG$-torsor. (This is an example of a linear modification, see \cite[\S 2]{P}.) In fact, since blowing-up commutes with \'etale localization, $ \mathcal{A}^{\rm bl}_{\mathbf{K}}$ is the blow-up of $\mathcal{A}^{\rm spl}_{\mathbf{K}} $ along the locus of its special fiber where $t \mathscr{F}_1=0$. The projection gives a smooth $\calG$-morphism
\[
q^{\rm reg}_{\mathbf{K}}: \wti{\mathcal{A}}^{\rm bl}_{\mathbf{K}} \longrightarrow {\rm M}^{\rm bl}
\]
which completes the local model diagram. Property (c) follows from Theorem \ref{MSemi} and properties (a) and (b) which imply that $ \mathcal{A}^{\rm bl}_{\mathbf{K}}$ and
$ {\rm M}^{\rm bl}$ are locally isomorphic for the \'etale topology. 
\end{proof} 
\begin{Corollary}
$\mathcal{A}^{\rm bl}_{\mathbf{K}}$ is the blow-up of $\mathcal{A}^{\rm spl}_{\mathbf{K}} $ along the locus of its special fiber where the deRham filtration $\mathscr{F}_1 = {\rm Fil}^0 (A) $ is annihilated by the action of the uniformizer $\pi$. 
\end{Corollary}
\begin{proof}
    It follows from the proof of the above theorem. 
\end{proof}
\begin{Remarks}
{\rm 
\begin{enumerate}
    \item From the above discussion, we can obtain a semi-stable integral model for the Shimura variety ${\rm Sh}_{\mathbf{K}'}(G, X)$ where $\mathbf{K}'=K^pK^\circ $. In this case, the corresponding local models $\Mloc$ of ${\rm Sh}_{\mathbf{K}'}(G, X)$ agree with the Pappas-Zhu local models $ \mathbb{M}_{K^\circ}(G,\mu_{r,s})$ for the local model triples $(G,\{\mu_{r,s}\},K^\circ)$. (See \cite[Theorem 1.2]{PZ} and \cite[\S 8]{PZ} for more details.)
    \item Similar results can be obtained for corresponding Rapoport-Zink formal schemes. (See \cite[\S 4]{HPR} for an example of this parallel treatment.)
\end{enumerate}
}
\end{Remarks}

\end{document}